\documentclass{scrartcl}
\usepackage[utf8]{inputenc}
\usepackage{amsmath,amsthm,amssymb,graphicx,tikz,caption,subcaption,xcolor,nicefrac,cleveref}
\usepackage{bm}
\usepackage{algorithm}
\usepackage[noend]{algorithmic}
\usepackage[textsize=footnotesize]{todonotes}
\usepackage{natbib}

\newtheorem{theorem}{Theorem}

\newtheorem{lemma}[theorem]{Lemma}
\newtheorem{corollary}[theorem]{Corollary}

\newtheorem{defn}[theorem]{Definition}
\newtheorem{assumption}[theorem]{Assumption}
\newtheorem{question}[theorem]{Question}

\newcommand{\cB}{\mathcal{B}}

\newcommand{\oA}{\overline{A}}
\newcommand{\oS}{\overline{S}}

\newcommand{\N}{\mathbb{N}}
\newcommand{\R}{\mathbb{R}}
\newcommand{\E}{\mathbb{E}}
\newcommand{\eps}{\epsilon}
\newcommand{\cD}{\mathcal{D}}
\newcommand{\cF}{\mathcal{F}}
\newcommand{\bone}{\mathbf{1}}

\newcommand{\of}{\overline{f}}
\newcommand{\bx}{\mathbf{x}}
\newcommand{\by}{\mathbf{y}}
\newcommand{\obx}{\overline{\bx}}
\newcommand{\ox}{\overline{x}}
\newcommand{\oy}{\overline{y}}
\newcommand{\oby}{\overline{\by}}
\newcommand{\ochi}{\overline{\chi}}

\DeclareMathOperator{\disc}{disc}

\newcommand{\cM}{\mathcal{M}}

\newcommand{\ignore}[1]{}

\allowdisplaybreaks

\title{Discrepancy Beyond Additive Functions \\ with Applications to Fair Division}

\author{
Alexandros Hollender\\University of Oxford
\and
Pasin Manurangsi\\Google Research
\and
Raghu Meka\\UCLA
\and
Warut Suksompong\\National University of Singapore
}

\date{\vspace{-10mm}}

\begin{document}
\maketitle

\begin{abstract}
We consider a setting where we have a ground set $\cM$ together with real-valued set functions $f_1, \dots, f_n$, and the goal is to partition $\cM$ into two sets $S_1,S_2$ such that $|f_i(S_1) - f_i(S_2)|$ is small for every~$i$. 
Many results in discrepancy theory can be stated in this form with the functions $f_i$ being additive. 
In this work, we initiate the study of the unstructured case where $f_i$ is \emph{not} assumed to be additive. 
We show that even without the additivity assumption, the upper bound remains at most $O(\sqrt{n \log n})$. 

Our result has implications on the fair allocation of indivisible goods.
In particular, we show that a consensus halving up to $O(\sqrt{n \log n})$ goods always exists for $n$ agents with monotone utilities. 
Previously, only an $O(n)$ bound was known for this setting.
\end{abstract}

\section{Introduction}

The area of combinatorial discrepancy traditionally studies questions of the following form: Given a hypergraph (or equivalently, set systems) $G = (V, E)$, find a $2$-coloring of the vertices such that the largest difference between the numbers of vertices of each color in a hyperedge is minimized. 
This minimum is referred to as the \emph{discrepancy} of the hypergraph. 
Discrepancy and its generalizations to questions in \emph{vector balancing} have been extensively studied in combinatorics with many applications \citep{BeckFi81,Spencer85, Banaszczyk98, Bansal10, Rot13, LovettMe15}. 
Of particular relevance is the seminal result of \citet{Spencer85}, which shows that the discrepancy of any hypergraph with at most $n$ hyperedges\footnote{We use $n$ to denote the number of hyperedges (instead of the number of vertices) in order to be consistent with the fair division literature, which we will discuss later (see, e.g., \citep{ManurangsiSu22}).} is at most $O(\sqrt{n})$, regardless of the number of vertices. 

Instead of considering \emph{the number of vertices in a hyperedge} of each color, we can consider any set function. 
Specifically, instead of hyperedges, we are given set functions $f_1, \dots, f_n$ and asked to find a coloring that minimizes the maximum difference of $f_i$ evaluated on the sets corresponding to the two colors.
This is formalized below,\footnote{Here and throughout, we write a subset and its indicator vector interchangeably when there is no ambiguity.} 
where $[r]:= \{1,\dots,r\}$ for any positive integer~$r$.

\begin{defn}
Let $\cF = \{f_1,\ldots,f_n: \{0,1\}^m \rightarrow \R\}$ be a family of functions.
The \emph{discrepancy} of $\cF$ is defined as
\begin{align*}
\disc(\cF) := \min_{\chi: [m] \to [2]} \max_{i \in [n]} |f_i(\chi^{-1}(1)) - f_i(\chi^{-1}(2))|.
\end{align*}
\end{defn}

Observe that in the classic setting of hypergraphs, $m$ is the number of vertices and each hyperedge $e$ corresponds to the function $f_e(S) = |S \cap e|$. 
Such a function $f_e$ is \emph{additive}, i.e., $f(S_1 \cup S_2) = f(S_1) + f(S_2)$ for any disjoint sets $S_1, S_2$.
Spencer's classical result extends beyond hypergraphs to the setting where each function $f_i(S) = \sum_{g \in S} a_{ig}$ is defined by coefficients $a_{ig} \in [-1,1]$. 
Note that in this more general setting, the functions $f_i$ are still additive.
Another additive case that has been investigated previously is when $f$ corresponds to a probability measure, for example, in geometric discrepancy \citep{Matousek99}.

A well-studied generalization is \emph{vector balancing}. Given vectors $v_1, \dots, v_m \in \mathbb{R}^n$ and a semi-norm $\phi: \mathbb{R}^n \rightarrow \mathbb{R}$, the objective is to find a coloring $\chi: [m] \rightarrow \{1, 2\}$ that minimizes
$$\phi\left(\sum_{g \in \chi^{-1}(1)} v_g - \sum_{g \in \chi^{-1}(2)} v_g \right).$$
For a detailed history of this problem, we refer the reader to the papers by \citet{Banaszczyk98} and \citet{Rot17} as well as the references therein. 
Spencer's theorem, for example, is an instance of this problem where the vectors $v_g$ have entries in $[-1,1]$ and $\phi$ is the $\ell_\infty$-norm.
Although the objective function $\phi(\cdot)$ is not required to be additive, an additive structure is still inherent to this formulation: vectors are first summed within each partition, and the semi-norm is then applied to the outcome. Most existing results in discrepancy theory apply to settings that are implicitly additive in this manner.

With this in mind, it is natural to extend discrepancy theory to functions $f_i$ that are \emph{not necessarily additive}. This generalization not only links to new problems in combinatorics and fair division, but also introduces significant challenges, as standard techniques like the \emph{partial coloring method} are no longer applicable (see \Cref{sec:conclusion_and_open_questions}).

\subsection{Our Contribution}

In this work, we extend the study of discrepancy by considering functions $f_i$ with as few assumptions as possible.
Specifically, we only assume that $f_i$ is \emph{$1$-Lipschitz}, i.e., changing the set $S$ by a single element can change the value of~$f_i$ by at most\footnote{Note that the value~$1$ is without loss of generality, since we can always rescale our function.} $1$. 
Under this assumption, we establish an $O(\sqrt{n \log n})$ bound on the discrepancy.

\begin{theorem}[Main Theorem (Informal)] \label{thm:main-informal}
Every $1$-Lipschitz family $\cF$ of $n$ functions has discrepancy $O(\sqrt{n \log n})$.
\end{theorem}

Note that in this setting, \citet{Spencer85} showed a bound of $O(\sqrt{n})$ for additive functions. 
When $m = \Omega(n)$, a tight lower bound of $\Omega(\sqrt{n})$ is also known. Thus, our result is tight up to a factor of $O(\sqrt{\log n})$; closing this gap remains an interesting question.
In the additive setting, a simpler argument yields a weaker $O(\sqrt{n \log n})$ bound: one can use linear programming (LP) feasibility to reduce the universe size to $n$ and then apply a random coloring. A primary challenge beyond the additive setting is that this reduction step is no longer straightforward. We show that a similar step is still possible, but it requires invoking the Borsuk--Ulam theorem from topology (in the form of \Cref{thm:consensus-div}) instead of a simple LP feasibility argument.

\paragraph{Fair Division.} 
As an application of our result, we show a bound of $O(\sqrt{n \log n})$ for consensus halving of indivisible goods when agents have monotone (but not necessarily additive) utilities. 
A \emph{consensus halving} is a partition of the goods into two parts such that every agent values both parts equally.
Our relaxation allows each agent's utilities for the two parts to differ by up to $O(\sqrt{n \log n})$ goods; see \Cref{sec:fair_div} for the formal definition.

\begin{theorem} \label{thm:halving}
For any $n$ agents with monotone utilities on any finite number of goods, there always exists a consensus halving up to $O(\sqrt{n \log n})$ goods.  
\end{theorem}

To the best of our knowledge, only an $O(n)$ bound was previously known for this general setting \citep{SimmonsSu03}.\footnote{Strictly speaking, \citet{SimmonsSu03} only considered the case of divisible goods. 
However, it is not difficult to extend their result to the case of indivisible goods by rounding---see, e.g., the proof of \Cref{lem:few-fractional}.} 
On the other hand, an $O(\sqrt{n})$ bound has been proven for the additive case \citep{ManurangsiSu22}. 
We also remark that consensus halving is stronger than the fundamental fairness notion of \emph{envy-freeness}, which requires that if the agents are partitioned into two groups, then each agent finds her group's bundle of goods to be at least as valuable as the other group's bundle.
Thus, as a corollary, we also obtain the existence of an allocation that is envy-free up to $O(\sqrt{n \log n})$ goods (\Cref{cor:ef}).

Our main discrepancy theorem (\Cref{thm:main-informal}) can be readily extended to the $k$-color version for any prime $k$---see \Cref{thm:main}. 
However, \Cref{thm:halving} does \emph{not} extend naturally to $k > 2$; we discuss the challenges associated with this question in \Cref{sec:conclusion_and_open_questions}.

\subsection{Concurrent Work}

In an independent and concurrent work,~\cite{DF25} study discrepancy in a non-additive setting similar to ours. Specifically, they prove \Cref{thm:main-informal} and its generalization to $k > 2$ colors. For $k > 2$, their results are stronger than ours (\Cref{thm:main}) in two aspects: their $k$ can be an arbitrary prime power---compared to a prime in our case---and their discrepancy bound is $O(\sqrt{n \log(nk)})$---compared to our bound of $O(\sqrt{nk \log(nk)})$. This gap of $O(\sqrt{k})$ between the two bounds can be significant for applications where $k$ is large; one such example discussed by \cite{DF25} is ``envy-freeness with subsidies'' for which one needs $n = k$. The advantage comes from their use of a necklace splitting theorem of \cite{JojicPZ21}, which gives a refined bound on the number of intervals of each bundle.
On the other hand, we use a consensus $1/k$-division result of \cite{FilosRatsikasHoSo21}, which only bounds the total number of intervals (see \Cref{thm:consensus-div}).

\cite{DF25} also prove the existence of a relaxation of consensus halving for indivisible goods. Although the result is similar to our \Cref{thm:halving}, their relaxation of consensus halving differs in that it allows \emph{transferring} $c$ items from one bundle to another, whereas our version, which is commonly used in fair division, only allows \emph{removing} $c$ items from one bundle (see \Cref{def:consesus-halving}). Since the latter notion is more stringent, our result implies theirs.

\section{Preliminaries}

For notational convenience, we will write $S \subseteq [m]$ and its indicator vector (i.e., $\bone_S \in \{0, 1\}^m$ such that $(\bone_S)_g = 1$ if and only if $g \in S$) interchangeably throughout this work.
Furthermore, we write $\oS$ to denote the complement of $S$ (i.e., $\oS := [m] \setminus S$).

We let $\Delta^{(k - 1)}$ denote the $(k - 1)$-simplex, i.e., $\{\bx \in [0, 1]^k \mid \sum_{j \in [k]} x_j = 1\}$. 
A \emph{$k$-coloring} of $[m]$ is a function $\chi: [m] \to [k]$; a \emph{fractional $k$-coloring} of $[m]$ is a function $\ochi: [m] \to \Delta^{(k - 1)}$. 
We say that $g \in [m]$ is \emph{integral} (with respect to $\ochi$) if $\ochi(g)$ is one of the vertices of the simplex; otherwise, we say that it is \emph{fractional}. 
For convenience, for every $j \in [k]$, we also define $\ochi_j \in [0, 1]^m$ by $(\ochi_j)_g := \ochi(g)_j$ for each $g \in [m]$.

For any $\obx \in [0, 1]^m$, we let $\cD_{\obx}$ denote the distribution over subsets of $[m]$ such that each $g \in [m]$ is included independently with probability $\ox_g$.
For a function $f: \{0, 1\}^m \to \R$, its \emph{multilinear extension} $\of: [0, 1]^m \to \R$ is defined by $\of(\obx) := \E_{\bx \sim \cD_{\obx}}[f(\bx)]$ for any $\obx \in [0, 1]^m$.

For any sets $S, T \subseteq [m]$, we use $d(S, T)$ to denote the Hamming distance between $S$ and~$T$. For a non-empty collection $\cM$ of subsets of $[m]$, we let $d(S, \cM) := \min_{T \in \cM} d(S, T)$.

We say that a function $f: \{0, 1\}^m \to \R$ is \emph{$1$-Lipschitz} if $|f(S) - f(S')| \leq 1$ for any sets $S, S'$ such that $d(S, S') \leq 1$. 
Moreover, we say that $f$ is \emph{monotone} if $f(S) \leq f(S')$ for any $S \subseteq S'$. 
A family $\cF$ of functions is said to be $1$-Lipschitz (resp., monotone) if all of its functions are $1$-Lipschitz (resp., monotone).

\subsection{Fair Division: Consensus $\boldsymbol{1/k}$-Division for Divisible Goods}
\label{subsec:prelim_consensus}

In the fair allocation of \emph{divisible goods}, also referred to as ``cake cutting'', a \emph{consensus $1/k$-division} is a partition of the cake into $k$ pieces such that each of the $n$ agents values each piece of the cake exactly the same as any other piece. The cake is simply an interval and the goal is to obtain a consensus $1/k$-division using only a few cuts.

When the utility functions of the agents are additive, it is known that a consensus $1/k$-division using at most $n(k-1)$ cuts exists for any~$k$ \citep{Alon87}. For utility functions that are not additive, this result is only known for prime $k$. The case $k = 2$ was shown by \citet{SimmonsSu03}, and the extension to prime numbers $k > 2$ by \citet{FilosRatsikasHoSo21}.\footnote{Specifically, this is stated as Theorem~6.5 in the extended version of Filos-Ratsikas et al.'s work.}

Let $\cB([0, 1])$ denote the Borel $\sigma$-algebra on the unit interval, let $\lambda$ denote the Lebesgue measure on the unit interval, and let $\Delta$ denote the symmetric difference.

\begin{assumption} \label{as:smooth-util}
Let $u: \cB([0, 1]) \to \R$. 
We assume that $u$ satisfies the following: For any $\eps > 0$, there exists $\delta > 0$ such that, for all $A, B \in \cB([0, 1])$ that satisfy $\lambda(A \Delta B) \leq \delta$, we have $|u(A) - u(B)| \leq \eps$.
\end{assumption}

\begin{theorem}[{\citep{SimmonsSu03,FilosRatsikasHoSo21}}] \label{thm:consensus-div}
For any prime number~$k$ and any utility functions $u_1, \dots, u_n: \cB([0, 1]) \to \R$ that satisfy Assumption~\ref{as:smooth-util}, it is possible to partition the unit interval into $k$ pieces $A_1, \dots, A_k$ using at most $n(k - 1)$ cuts so that $u_i(A_j) = u_i(A_\ell)$ for all $i \in [n]$ and $j, \ell \in [k]$.
\end{theorem}

We note that this theorem does \emph{not} require the utility functions to be additive or non-negative, and its proof relies on (a generalization of) the Borsuk--Ulam theorem.

\subsection{Concentration Inequality}

Finally, we state the following well-known inequality of McDiarmid, which is sometimes referred to as the ``bounded difference method''.

\begin{theorem}[\citep{McDiarmid89}] \label{thm:bounded-diff}
Let $q\in\N$, and let $h: \{0, 1\}^q \to \R$ be any $1$-Lipschitz function. 
For any independent random variables $X_1, \dots, X_q$, let $Y := h(X_1, \dots, X_q)$. 
Then, for any $t > 0$, we have
\begin{align*}
\Pr[|Y - \E[Y|| \geq t] \leq 2\exp\left(-\frac{2t^2}{ q}\right).
\end{align*}
\end{theorem}

\section{Discrepancy Beyond Additive Functions}
\label{sec:discrepancy_of_monotone_and_lipschitz_functions}

The definition of discrepancy for $k\ge 2$ colors can be defined as follows. Again, we note that this generalizes the standard definition for hypergraphs \citep{DoerrSr03}.

\begin{defn}
Let $\cF = \{f_1, \dots, f_n\}$ be a family of functions $f_1, \dots, f_n: \{0, 1\}^{m} \to \R$. For a $k$-coloring $\chi: [m] \to [k]$, its \emph{discrepancy (with respect to $\cF$)} is defined as
\begin{align*}
\disc(\cF, \chi) := \max_{i \in [n]} \max_{j, \ell \in [k]} |f_i(\chi^{-1}(j)) - f_i(\chi^{-1}(\ell))|.
\end{align*}
Then, the \emph{$k$-color discrepancy} of $\cF$ is
\begin{align*}
\disc(\cF, k) := \min_{\chi: [m] \to [k]} \disc(\cF, \chi).
\end{align*}
\end{defn}

We state our main bound below.
Note that when $k$ is constant (e.g., $k = 2$ as in \Cref{thm:main-informal}), this bound is $O(\sqrt{n\log n})$.

\begin{theorem}[Main Theorem] \label{thm:main}
For any prime number~$k$ and any $1$-Lipschitz family $\cF$ of $n$ functions, it holds that $$\disc(\cF, k) < \sqrt{2n(k - 1) \cdot \left(1 + \ln (n k)\right)}.$$
\end{theorem}

When the functions $f_i$ are additive, a lower bound of $\Omega(\sqrt{n})$ has been shown by \citet{MM25}. 
Thus, our bound is tight up to a factor of $O(\sqrt{k \log(nk)})$.
We also remark that a random coloring (together with a concentration argument) gives an upper bound of $O_k(\sqrt{m \log n})$. Therefore, our contribution can be viewed as removing the dependency on $m$, which is crucial for applications in fair division (see \Cref{sec:fair_div}) where $m$ is often large.

Before we proceed to the formal proof, let us give some high-level intuition. 
The bound in \Cref{thm:main} can be easily shown in the additive setting as follows. 
First, we write a linear program with the constraints $f_i(\chi^{-1}(j)) = f_i(\chi^{-1}(\ell))$. 
Since there are only $nk$ such (non-redundant) constraints, we can solve for an extreme point solution that has at most $nk$ non-integral variables.
This corresponds to a fractional coloring with at most $nk$ fractional variables.
We round each of these fractional variables independently; a concentration argument then yields the desired bound.

Since we are dealing with non-additive functions, we cannot use linear programs. Nevertheless, we can derive the following analogous lemma, which is a corollary of \Cref{thm:consensus-div}. 
The idea of the proof is to place the $m$ elements on the unit interval, with each element corresponding to an interval of size~$1/m$.
We then apply \Cref{thm:consensus-div} to find the desired cuts. 

\begin{lemma} \label{lem:few-fractional}
For any prime number $k$ and any $1$-Lipschitz functions $f_1, \dots, f_n: \{0, 1\}^m \to \R$, there exists a fractional $k$-coloring $\ochi: [m] \to \Delta^{(k-1)}$ with at most $n(k - 1)$ fractional elements such that $\of_i(\ochi_j) = \of_i(\ochi_\ell)$ for all $i \in [n]$ and $j, \ell \in [k]$.
\end{lemma}

\begin{proof}
For each $i \in [n]$, we define the utility function $u_i: \cB([0, 1]) \to \R$ as follows. 
For any $A \in \cB([0, 1])$, let $\obx^A \in [0, 1]^m$ be defined by $$\ox^A_g = m \cdot \lambda\left(A \cap \left[\frac{g-1}{m}, \frac{g}{m}\right]\right)$$ for each $g \in [m]$. 
We then let $u_i(A) := \of_i(\obx^A)$ for any $A \in \cB([0, 1])$. 
Since $f_i$ is $1$-Lipschitz, $u_i$ satisfies\footnote{See \Cref{app:lipschitz} for a more detailed explanation.} Assumption~\ref{as:smooth-util} with $\delta = \eps/m$. 
Thus, we can apply \Cref{thm:consensus-div} to obtain a partition $A_1, \dots, A_k$ of $[0, 1]$ resulting from at most $n(k - 1)$ cuts such that $u_i(A_j) = u_i(A_\ell)$ for all $i \in [n]$ and $j, \ell \in [k]$. 
We then define the fractional $k$-coloring~$\ochi$ by $\ochi(g)_j := \ox^{A_j}_g$. 
By definition of $u_i$, this immediately implies that $\of_i(\ochi_j) = \of_i(\ochi_\ell)$ for all $i \in [n]$ and $j, \ell \in [k]$. 
Finally, observe that an element~$g$ can be fractional only if at least one of the cuts belongs to $\left(\frac{g-1}{m}, \frac{g}{m}\right)$. 
Hence, there are at most $n(k - 1)$ fractional elements, as desired.
\end{proof}

We can now prove \Cref{thm:main} by using \Cref{lem:few-fractional} together with (independent) randomized rounding and a concentration argument.

\begin{proof}[Proof of \Cref{thm:main}]
Let $\cF = \{f_1,\dots,f_n\}$.
Applying \Cref{lem:few-fractional} yields a fractional $k$-coloring $\ochi: [m] \to \Delta^{(k-1)}$ with at most $n(k - 1)$ fractional elements such that $\of_i(\ochi_j) = \of_i(\ochi_\ell)$ for all $i \in [n]$ and $j, \ell \in [k]$. 
Let $S$ denote the set of fractional elements with respect to~$\ochi$. 
Consider the following (randomized) $k$-coloring $\chi: [m] \to [k]$: independently for every $g \in [m]$, let $\chi(g) = j$ with probability $\ochi(g)_j$ for each $j \in [k]$. 
Let $t := \sqrt{\frac{n(k - 1) \cdot \left(1 + \ln (n k)\right)}{2}}$. 
We will show that
\begin{align} \label{eq:concen-applied}
\Pr\left[\left|f_i(\chi^{-1}(j)) - \of_i(\ochi_j)\right| \geq t\right] < \frac{1}{nk}
\end{align}
for all $i \in [n]$ and $j \in [k]$.

Before we prove Inequality~\eqref{eq:concen-applied}, let us first show how it implies the claimed bound on $\disc(\cF, k)$. 
By \eqref{eq:concen-applied} and the union bound, we have that with positive probability, 
\begin{align*}
\left|f_i(\chi^{-1}(j)) - \of_i(\ochi_j)\right| < t
\end{align*}
for all $i \in [n]$ and $j \in [k]$ simultaneously. 
Recall that $\of_i(\ochi_j) = \of_i(\ochi_\ell)$ for all $i \in [n]$ and $j, \ell \in [k]$.
Thus, when the event above occurs, we have that $|f_i(\chi^{-1}(j)) - f_i(\chi^{-1}(\ell))| < 2t$ for all $i \in [n]$ and $j, \ell \in [k]$. 
This implies that $\disc(\cF, k) < 2t$, as desired.

We now prove Inequality~\eqref{eq:concen-applied}. 
Consider fixed $i \in [n]$ and $j \in [k]$. 
For each $g \in S$, let $X_g$ be an indicator random variable that indicates whether $\chi(g) = j$ (note that, for all $g \notin S$, $\chi(g)$ is already fixed). 
We can thus view $f_i(\chi^{-1}(j))$ as a function $h((X_g)_{g \in S})$. 
Since $f_i$ is $1$-Lipschitz, so is $h$. 
Furthermore, the random variables $X_g$ are independent. 
Let $Y := h((X_g)_{g \in S})$. 
Applying \Cref{thm:bounded-diff} with $q := |S| \leq n(k - 1)$, we get
\begin{align*}
\Pr\left[|f_i(\chi^{-1}(j)) - \of_i(\ochi_j)| \geq t\right] 
= \Pr\left[|Y - \E[Y]| \geq t\right]
\leq 2\exp\left(-\frac{2t^2}{q}\right)
< \frac{1}{nk},
\end{align*}
where the last inequality follows from our choice of $t$. This concludes our proof.
\end{proof}

\section{From Discrepancy to Fair Division}
\label{sec:fair_div}

In this section, we derive the implications of our discrepancy result for fair division.
Consider a setting with $n$~agents, where each agent~$i$ has a monotone utility function $u_i$ on a set of goods $[m]$.
We define an approximate version of consensus halving, which has been studied for additive utility functions \citep{ManurangsiSu22,CaragiannisLaSh25,MM25}.

\begin{defn} \label{def:consesus-halving}
For monotone utility functions $u_1, \dots, u_n$, a set of goods $[m]$, and $c \in \N$, we say that a partition $\{A_1, A_2\}$ of the $m$~goods is a \emph{consensus halving up to $c$ goods} if, for each $i \in [n]$, there exist $R_1 \subseteq A_1$ and $R_2 \subseteq A_2$ with $|R_1|,|R_2|\le c$ such that $u_i(A_1) \geq u_i(A_2 \setminus R_2)$ and $u_i(A_2) \geq u_i(A_1 \setminus R_1)$.
\end{defn}

We are now ready to prove that a consensus halving up to $O(\sqrt{n \log n})$ goods always exists for agents with monotone utilities (\Cref{thm:halving}). 
Our strategy here is, for every agent $i$, to define a ``valid set'' $\cM_i \subseteq \{0, 1\}^m$ such that, if both $A$ and $\oA$ are ``close'' to $\cM_i$ in terms of Hamming distance (for all $i \in [n]$), then $\{A, \oA\}$ forms a consensus halving up to $O(\sqrt{n \log n})$ goods. 
Once we have identified such a valid set $\cM_i$, we simply let $f_i(A)$ be the distance from $A$ to $\cM_i$. 
This function is 1-Lipschitz and allows us to invoke our result from the previous section to conclude the proof. This idea is formalized below.

\begin{proof}[Proof of \Cref{thm:halving}]
Let $c := 2\sqrt{2n(k - 1) \cdot \left(1 + \ln (n k)\right)}$ be twice the bound in \Cref{thm:main}, where $k = 2$. 
For each $i \in [n]$, let us first define a modified utility function $u'_i: \{0, 1\}^m \to \R$ by
\begin{align*}
u'_i(S) := \min_{T \subseteq S,\, |T| \leq c/2} u_i(S \setminus T)
\end{align*}
for each $S\subseteq [m]$.
Then, we define $\cM_i \subseteq \{0, 1\}^m$ by
\begin{align*}
\cM_i := \left\{S \in \{0, 1\}^m \,\middle|\, u'_i(S) \geq u'_i(\oS)\right\},
\end{align*}
and $f_i: \{0, 1\}^m \to \R$ by
\begin{align*}
f_i(S) := d(S, \cM_i).
\end{align*}
Note that the distance is well-defined since $\cM_i$ is not empty; in fact, for each set~$S$, at least one of $S$ and $\oS$ belongs to $\cM_i$. 
Furthermore, $\cM_i$ is monotone, i.e., if $S$ belongs to~$\cM_i$, then so does any $S' \supseteq S$.

Observe that $f_i$ is $1$-Lipschitz. 
Thus, we may apply \Cref{thm:main} with $k = 2$ to find a coloring $\chi: [m] \to [2]$ such that the following holds for all $i \in [n]$:
\begin{align} \label{eq:diff-bound}
\left|f_i(\chi^{-1}(1)) - f_i(\chi^{-1}(2))\right| < \frac{c}{2}.
\end{align}
We claim that the allocation $\{A_1, A_2\} := \{\chi^{-1}(1), \chi^{-1}(2)\}$ is a consensus halving up to $c$ goods. To see this, consider any $i \in [n]$. 
Suppose without loss of generality that $u'_i(A_1) \geq u'_i(A_2)$. 
By definition of $u'_i$, there exist $T_1, T_2$ with $|T_1|,|T_2| \le c/2$ such that $u'_i(A_1) = u_i(A_1 \setminus T_1)$ and $u'_i(A_2) = u_i(A_2 \setminus T_2)$. 
On the one hand, we immediately have
\begin{align*}
u_i(A_1) \geq u_i(A_1 \setminus T_1) = u'_i(A_1) \geq u'_i(A_2) = u_i(A_2 \setminus T_2),
\end{align*}
where the first inequality follows from the monotonicity of $u_i$.

On the other hand, $u'_i(A_1) \geq u'_i(A_2)$ implies that $A_1 \in \cM_i$, and thus $f_i(A_1) = 0$. 
Inequality~\eqref{eq:diff-bound} then implies that $f_i(A_2) < c/2$, that is, $d(A_2, \cM_i) < c/2$. 
Since $\cM_i$ is monotone, this means that there exists $T'_2 \subseteq \overline{A_2} = A_1$ of size less than $c/2$ such that $A_2 \cup T'_2 \in \cM_i$.
We thus have
\begin{align*}
u_i(A_2) \geq u'_i(A_2 \cup T'_2) \geq u'_i(\overline{A_2 \cup T'_2}) = u'_i(A_1 \setminus T'_2),
\end{align*}
where the first inequality follows from the definition of $u'_i$ and the fact that $|T'_2| < c/2$, and the second from $A_2 \cup T'_2 \in \cM_i$. 
Now, by definition of $u'_i$, there exists $T'_1$ of size at most $c/2$ such that $u'_i(A_1 \setminus T'_2) = u_i(A_1 \setminus (T'_1 \cup T'_2))$. 
Plugging this into the chain of inequalities above, we get
\begin{align*}
u_i(A_2) \geq u_i(A_1 \setminus (T'_1 \cup T'_2)).
\end{align*}
Since $|T'_1| \le c/2$ and $|T'_2| < c/2$, we have $|T'_1 \cup T'_2| < c$. 
Thus, $\{A_1, A_2\}$ is a consensus halving up to $c$ goods, as claimed.
\end{proof}

Finally, we state a corollary on envy-freeness, a fundamental notion in fair division.
Suppose that the $n$~agents are partitioned into two groups arbitrarily, and consider an allocation $(A_1,A_2)$ of the $m$ goods to the two groups, where $A_j$ is allocated to group~$j$ for each $j \in [2]$.
For $c\in\mathbb{N}$, the allocation $(A_1,A_2)$ is said to be \emph{envy-free up to $c$ goods} if for each $j\in[2]$ and each agent~$i$ in group~$j$, there exists $R_{3-j}\subseteq A_{3-j}$ with $|R_{3-j}| \le c$ such that $u_i(A_j) \ge u_i(A_{3-j}\setminus R_{3-j})$.
Observe that a consensus halving up to $c$ goods directly gives rise to an allocation that is envy-free up to $c$ goods, by allocating the two bundles to the two groups arbitrarily.
Hence, \Cref{thm:halving} yields the following corollary.

\begin{corollary} \label{cor:ef}
For any $n$ agents with monotone utilities partitioned into two groups, there always exists an allocation that is envy-free up to $O(\sqrt{n\log n})$ goods.
\end{corollary}

\section{Open Questions}
\label{sec:conclusion_and_open_questions}

A conceptual contribution of this work is to introduce new questions in discrepancy theory for non-additive functions. We highlight two such questions and describe a fundamental challenge inherent in these generalizations: the inapplicability of standard techniques.

\paragraph{Closing the $\boldsymbol{\sqrt{\log n}}$ Gap.} 
An obvious open question is to close the $O_k(\sqrt{\log n})$ gap in our main bound (\Cref{thm:main}).
This gap is present even in the case $k = 2$ (\Cref{thm:main-informal}).

\begin{question} Is the discrepancy of every $1$-Lipschitz family $\mathcal{F} = \{f_1,\ldots,f_n: \{0,1\}^m \rightarrow \R\}$ of functions bounded as $O(\sqrt{n})$?
\end{question}

An affirmative answer would be a far-reaching generalization of Spencer's classical result. However, a fundamental hurdle is that most proofs of Spencer's theorem and its variants rely on the \emph{partial coloring method}, introduced by \citet{B81} and \citet{Gluskin89} and used extensively throughout discrepancy theory \citep{Spencer85, Bansal10, Rot13, Rot17, LovettMe15, BJM24}.

The partial coloring method seeks a coloring $\chi:[m] \rightarrow \{0,1,2\}$ that partitions the ground set into three parts.\footnote{Some works consider a continuous third color, but this does not avoid the main issue discussed below.} The goal is to color a large subset of elements (i.e., maximize $|\chi^{-1}(\{1,2\})|$) while minimizing the induced discrepancy,
$$\max_{i \in [n]} |f_i(\chi^{-1}(1)) - f_i(\chi^{-1}(2))|.$$
The algorithm then recurses on the uncolored elements in $\chi^{-1}(0)$. For additive functions or vector balancing, the total discrepancy can be controlled via the triangle inequality. 
Unfortunately, this approach fails for general non-additive $1$-Lipschitz functions.\footnote{For instance, one can construct vectors where two disjoint partial colorings both induce a small discrepancy, but their union induces an arbitrarily large discrepancy.} 
For example, this method is ineffective even for functions of the form $f_i(S) = \|\sum_{g \in S} a_{ig}\|_\infty$, where each $a_{ig}$ is a vector with $\|a_{ig}\|_\infty \leq 1$. Thus, new techniques are required.

Although the work of \citet{Banaszczyk98} and the recent breakthroughs by \citet{BJ25a, BJ25b} do not directly use the partial coloring lemma, their analyses still crucially rely on bounding the discrepancy by coloring subsets and applying the triangle inequality.

\paragraph{Sparse Families of Functions.} Another notable question is to study the \emph{sparse} or \emph{Beck--Fiala setting}, introduced by \citet{BeckFi81}, where each element affects only a few functions.

\begin{defn}
We say that an element $g \in [m]$ is \emph{relevant} for $f: \{0, 1\}^m \to \R$ if $f$ depends on the assignment of $g$. That is, there exists $S \subseteq [m]$ such that $f(S) \ne f(S \cup \{g\})$. Let $R(f)$ denote the set of all relevant elements of $f$.

For $t\in\N$, a family of functions $\mathcal{F} = \{f_1,\ldots,f_n: \{0,1\}^m \rightarrow \R\}$ is \emph{$t$-sparse} if each $g \in [m]$ is relevant for at most $t$ functions (i.e., $|\{i \in [n] \mid g \in R(f_i)\}| \leq t$).
\end{defn}
In the classic hypergraph setting, this means that each vertex belongs to at most $t$ hyperedges. 
A seminal result by \citet{BeckFi81} shows that for additive $t$-sparse functions, the discrepancy is at most $2t-1$. 
A major conjecture is whether this can be improved to $O(\sqrt{t})$, a problem almost resolved by the recent work of \citet{BJ25a, BJ25b}.
In contrast, the discrepancy of general Lipschitz $t$-sparse functions is wide open.

\begin{question}
    Let $\mathcal{F} = \{f_1,\ldots,f_n:\{0,1\}^m \rightarrow \R\}$ be a collection of $1$-Lipschitz $t$-sparse functions. Is $\mathrm{disc}(\mathcal{F}) = O_t(1)$?
\end{question}

That is, can the discrepancy be bounded by a function of $t$ alone, independent of $n$ and $m$? This is open even for $t=2$. A key challenge remains the inapplicability of the partial coloring method or iterative coloring approaches used in the original work of \citet{BeckFi81}.

Using our main theorem, we provide an improved bound for such $t$-sparse families.

\begin{theorem}
    Let $\mathcal{F} = \{f_1,\ldots,f_n:\{0,1\}^m \rightarrow \R\}$ be a collection of $1$-Lipschitz $t$-sparse functions. Then, $\mathrm{disc}(\mathcal{F}) = O\left( (mt)^{1/3} \cdot (\log(mt))^{1/3}\right)$.
\end{theorem}

\begin{proof}
Let $c := (mt)^{1/3} \cdot (\log(mt))^{1/3}$, and let $\cF' := \{f_i \in \cF \mid R(f_i) \geq c\}$. 
Since $\cF$ is $t$-sparse, we have $|\cF'| \leq \frac{mt}{c}$. 
Applying \Cref{thm:main} on $\cF'$ yields a 2-coloring $\chi: [m] \to [2]$ with $\disc(\cF', \chi) \leq O\left(\sqrt{\frac{mt}{c} \cdot \log\left(\frac{mt}{c}\right)}\right) = O(c)$.
Finally, observe that any $f_i \in \cF \setminus \cF'$ satisfies $|f_i(\chi^{-1}(1)) - f_i(\chi^{-1}(2))| \leq |R(f_i)| < c$. Thus, $\disc(\cF, \chi) \leq O(c)$.  
\end{proof}

When $m = \Theta(n)$, this bound is a polynomial improvement over the general $\Omega(\sqrt{n})$ lower bound that holds without a sparsity assumption.

\vspace{2mm}
Finally, we outline two additional questions that our work suggests.

\paragraph{Extending Consensus Division to $\boldsymbol{k > 2}$ Parts.} Our proof of \Cref{thm:halving} is specific to $k = 2$, since we can create the (Lipschitz) functions $f_i$ that capture consensus division in that case. 
However, for general~$k$, it is unclear how to achieve this. 
Indeed, given a single part of the partition, one cannot even determine whether the partition constitutes a consensus division: it is necessary to obtain all the remaining $k - 1$ parts as well in order to determine this. 
It would be useful to understand if our proof can be extended to tackle this issue, or whether a completely new technique is required.

\paragraph{Computational Efficiency.} Our main theorems (\Cref{thm:main,thm:halving}) do \emph{not} come with efficient algorithms, unlike in the additive case. 
However, this might not be a coincidence: an approximate version of consensus halving for divisible goods (in the sense of \Cref{thm:consensus-div} with $k = 2$) is known to be PPA-complete \citep{FilosRatsikasGo18}. 
It remains an interesting open problem whether we can also prove such a hardness result for consensus halving up to $c$ goods, which would explain the lack of efficient algorithms in our setting. 

\vspace{2mm}
More broadly, we believe that discrepancy beyond additive functions is an important extension of the standard model and an exciting direction to pursue for future research.

\bibliography{refs}

\begin{thebibliography}{25}
\providecommand{\natexlab}[1]{#1}
\providecommand{\url}[1]{\texttt{#1}}
\expandafter\ifx\csname urlstyle\endcsname\relax
  \providecommand{\doi}[1]{doi: #1}\else
  \providecommand{\doi}{doi: \begingroup \urlstyle{rm}\Url}\fi

\bibitem[Alon(1987)]{Alon87}
Noga Alon.
\newblock Splitting necklaces.
\newblock \emph{Advances in Mathematics}, 63\penalty0 (3):\penalty0 247--253,
  1987.

\bibitem[Banaszczyk(1998)]{Banaszczyk98}
Wojciech Banaszczyk.
\newblock Balancing vectors and {G}aussian measures of $n$-dimensional convex
  bodies.
\newblock \emph{Random Structures and Algorithms}, 12\penalty0 (4):\penalty0
  351--360, 1998.

\bibitem[Bansal(2010)]{Bansal10}
Nikhil Bansal.
\newblock Constructive algorithms for discrepancy minimization.
\newblock In \emph{Proceedings of the 51st Annual {IEEE} Symposium on
  Foundations of Computer Science (FOCS)}, pages 3--10, 2010.

\bibitem[Bansal and Jiang(2025{\natexlab{a}})]{BJ25a}
Nikhil Bansal and Haotian Jiang.
\newblock An improved bound for the {Beck-Fiala} conjecture.
\newblock \emph{CoRR}, abs/2508.01937, 2025{\natexlab{a}}.

\bibitem[Bansal and Jiang(2025{\natexlab{b}})]{BJ25b}
Nikhil Bansal and Haotian Jiang.
\newblock Decoupling via affine spectral-independence: {Beck-Fiala} and
  {Koml\'os} bounds beyond {Banaszczyk}.
\newblock \emph{CoRR}, abs/2508.03961, 2025{\natexlab{b}}.

\bibitem[Bansal et~al.(2023)Bansal, Jiang, and Meka]{BJM24}
Nikhil Bansal, Haotian Jiang, and Raghu Meka.
\newblock Resolving matrix {Spencer} conjecture up to poly-logarithmic rank.
\newblock In \emph{Proceedings of the 55th Annual {ACM} Symposium on Theory of
  Computing (STOC)}, pages 1814--1819, 2023.

\bibitem[Beck(1981)]{B81}
J{\'o}zsef Beck.
\newblock Roth’s estimate of the discrepancy of integer sequences is nearly
  sharp.
\newblock \emph{Combinatorica}, 1\penalty0 (4):\penalty0 319--325, 1981.

\bibitem[Beck and Fiala(1981)]{BeckFi81}
J{\'{o}}zsef Beck and Tibor Fiala.
\newblock ``{I}nteger-making'' theorems.
\newblock \emph{Discrete Applied Mathematics}, 3\penalty0 (1):\penalty0 1--8,
  1981.

\bibitem[Caragiannis et~al.(2025)Caragiannis, Larsen, and
  Shyam]{CaragiannisLaSh25}
Ioannis Caragiannis, Kasper~Green Larsen, and Sudarshan Shyam.
\newblock A new lower bound for multicolor discrepancy with applications to
  fair division.
\newblock In \emph{Proceedings of the 18th International Symposium on
  Algorithmic Game Theory (SAGT)}, pages 228--246, 2025.

\bibitem[Doerr and Srivastav(2003)]{DoerrSr03}
Benjamin Doerr and Anand Srivastav.
\newblock Multicolour discrepancies.
\newblock \emph{Combinatorics, Probability and Computing}, 12\penalty0
  (4):\penalty0 365--399, 2003.

\bibitem[{Dupre la Tour} and Fujii(2025)]{DF25}
Max {Dupre la Tour} and Kaito Fujii.
\newblock Discrepancy and fair division for non-additive valuations.
\newblock \emph{CoRR}, abs/2509.16802, 2025.

\bibitem[Filos-Ratsikas and Goldberg(2018)]{FilosRatsikasGo18}
Aris Filos-Ratsikas and Paul~W. Goldberg.
\newblock Consensus halving is {PPA}-complete.
\newblock In \emph{Proceedings of the 50th Annual ACM SIGACT Symposium on
  Theory of Computing (STOC)}, pages 51--64, 2018.

\bibitem[Filos-Ratsikas et~al.(2021)Filos-Ratsikas, Hollender, Sotiraki, and
  Zampetakis]{FilosRatsikasHoSo21}
Aris Filos-Ratsikas, Alexandros Hollender, Katerina Sotiraki, and Manolis
  Zampetakis.
\newblock A topological characterization of modulo-$p$ arguments and
  implications for necklace splitting.
\newblock In \emph{Proceedings of the 32nd ACM-SIAM Symposium on Discrete
  Algorithms (SODA)}, pages 2615--2634, 2021.
\newblock Extended version available as arXiv:2003.11974v2.

\bibitem[Gluskin(1989)]{Gluskin89}
Efim~Davydovich Gluskin.
\newblock Extremal properties of orthogonal parallelepipeds and their
  applications to the geometry of {B}anach spaces.
\newblock \emph{Mathematics of the USSR-Sbornik}, 64\penalty0 (1):\penalty0 85,
  1989.

\bibitem[Joji\'{c} et~al.(2021)Joji\'{c}, Panina, and
  \v{Z}ivaljevi\'{c}]{JojicPZ21}
Du\v{s}ko Joji\'{c}, Gaiane Panina, and Rade~T. \v{Z}ivaljevi\'{c}.
\newblock Splitting necklaces, with constraints.
\newblock \emph{{SIAM} Journal on Discrete Mathematics}, 35\penalty0
  (2):\penalty0 1268--1286, 2021.

\bibitem[Lovett and Meka(2015)]{LovettMe15}
Shachar Lovett and Raghu Meka.
\newblock Constructive discrepancy minimization by walking on the edges.
\newblock \emph{{SIAM} Journal on Computing}, 44\penalty0 (5):\penalty0
  1573--1582, 2015.

\bibitem[Manurangsi and Meka(2025)]{MM25}
Pasin Manurangsi and Raghu Meka.
\newblock Tight lower bound for multicolor discrepancy.
\newblock \emph{CoRR}, abs/2504.18489, 2025.

\bibitem[Manurangsi and Suksompong(2022)]{ManurangsiSu22}
Pasin Manurangsi and Warut Suksompong.
\newblock Almost envy-freeness for groups: Improved bounds via discrepancy
  theory.
\newblock \emph{Theoretical Computer Science}, 930:\penalty0 179--195, 2022.

\bibitem[Matoušek(1999)]{Matousek99}
Jiří Matoušek.
\newblock \emph{Geometric Discrepancy: An Illustrated Guide}.
\newblock Springer, 1999.

\bibitem[McDiarmid(1989)]{McDiarmid89}
Colin McDiarmid.
\newblock On the method of bounded differences.
\newblock In Johannes Siemons, editor, \emph{Surveys in Combinatorics, 1989:
  Invited Papers at the Twelfth British Combinatorial Conference}, pages
  148--188. Cambridge University Press, 1989.

\bibitem[Roch(2024)]{roch_mdp_2024}
Sebastien Roch.
\newblock \emph{Modern Discrete Probability: An Essential Toolkit}.
\newblock Cambridge University Press, 2024.

\bibitem[Rothvoss(2013)]{Rot13}
Thomas Rothvoss.
\newblock Approximating bin packing within ${O}(\log {OPT} \cdot \log \log
  {OPT})$ bins.
\newblock In \emph{Proceedings of the 54th IEEE Annual Symposium on Foundations
  of Computer Science (FOCS)}, pages 20--29, 2013.

\bibitem[Rothvoss(2017)]{Rot17}
Thomas Rothvoss.
\newblock Constructive discrepancy minimization for convex sets.
\newblock \emph{SIAM Journal on Computing}, 46\penalty0 (1):\penalty0 224--234,
  2017.

\bibitem[Simmons and Su(2003)]{SimmonsSu03}
Forest~W. Simmons and Francis~Edward Su.
\newblock Consensus-halving via theorems of {B}orsuk--{U}lam and {T}ucker.
\newblock \emph{Mathematical Social Sciences}, 45\penalty0 (1):\penalty0
  15--25, 2003.

\bibitem[Spencer(1985)]{Spencer85}
Joel Spencer.
\newblock Six standard deviations suffice.
\newblock \emph{Transactions of the American Mathematical Society},
  289\penalty0 (2):\penalty0 679--706, 1985.

\end{thebibliography}
\bibliographystyle{plainnat}

\appendix

\section{Lipschitzness of the Multilinear Extension}
\label{app:lipschitz}

We say that a function $h: [0, 1]^m \to \R$ is \emph{$1$-Lipschitz} if $|h(\bx) - h(\by)| \leq \|\bx - \by\|_1 $ for all $\bx, \by \in [0, 1]^m$. 
We state the following folklore result that, if a function is $1$-Lipschitz, then so is its multilinear extension.

\begin{lemma} \label{lem:multi-lipsc}
Let $f: \{0, 1\}^m \to \R$ be any $1$-Lipschitz function. 
Then, its multilinear extension $\of$ is also $1$-Lipschitz.
\end{lemma}

Before we prove the lemma, let us explain how it implies that the utility function $u_i$ constructed in the proof of \Cref{lem:few-fractional} satisfies Assumption~\ref{as:smooth-util} with $\delta = \eps/m$. 
To see this, consider any $A, B \in \cB([0,1])$ such that $\lambda(A \Delta B) \leq \delta$. We have 
\begin{align*}
\|\obx^A - \obx^B\|_1 = m \cdot \sum_{g \in [m]} \left|\lambda\left(A \cap \left[\frac{g-1}{m}, \frac{g}{m}\right]\right) - \lambda\left(B \cap \left[\frac{g-1}{m}, \frac{g}{m}\right]\right)\right| \leq m \cdot \delta = \eps.
\end{align*}
Applying \Cref{lem:multi-lipsc}, we immediately get that $|u_i(A) - u_i(B)| = |f_i(\obx^A) - f_i(\obx^B)| \leq \eps$.

We now prove \Cref{lem:multi-lipsc}. 
We will use standard probability definitions and notations; see, e.g., \citep{roch_mdp_2024}.

\begin{proof}[Proof of \Cref{lem:multi-lipsc}]
Consider any $\obx, \oby \in [0, 1]^m$. Recall that $\of(\obx) := \E_{\bx \sim \cD_{\obx}}[f(\bx)]$ and $\of(\oby) := \E_{\by \sim \cD_{\oby}}[f(\by)]$, where $\cD_{\obx}, \cD_{\oby}$ are product distributions whose expectations are $\obx, \oby$ respectively. 
For each $g \in [m]$, since $x_g, y_g$ are simply Bernoulli random variables with mean $\ox_g, \oy_g$ respectively, there is a coupling  $(x_g, y_g)$ such that $\Pr_{(x_g, y_g)}[x_g \ne y_g] = |\ox_g - \oy_g|$. Thus, we have
\begin{align*}
\left|\of(\obx) - \of(\oby)\right|
&= \left|\E_{x_1, \dots, x_m}[f(x_1, \dots, x_m)] - \E_{y_1, \dots, y_m}[f(y_1, \dots, y_m)]\right| \\
&= \left|\E_{(x_1, y_1), \dots, (x_m, y_m)}[f(x_1, \dots, x_m)] - \E_{(x_1, y_1), \dots, (x_m, y_m)}[f(y_1, \dots, y_m)]\right| \\
&= \left|\E_{(x_1, y_1), \dots, (x_m, y_m)}[f(x_1, \dots, x_m) - f(y_1, \dots, y_m)]\right| \\
&\leq \E_{(x_1, y_1), \dots, (x_m, y_m)}[\left|f(x_1, \dots, x_m) - f(y_1, \dots, y_m)\right|] \\
&\leq \E_{(x_1, y_1), \dots, (x_m, y_m)}[d((x_1, \dots, x_m), (y_1, \dots, y_m))] \\
&= \E_{(x_1, y_1), \dots, (x_m, y_m)}\left[\sum_{g\in[m]} \bone[x_g \ne y_g]\right] \\
&= \sum_{g\in[m]} \Pr_{(x_g, y_g)}[x_g \ne y_g] \\
&= \sum_{g \in [m]} |\ox_g - \oy_g| \\
&= \|\obx - \oby\|_1,
\end{align*}
where the first inequality follows from the triangle inequality and the second holds because $f$ is $1$-Lipschitz.
\end{proof}

\end{document}